\documentclass[12pt, a4paper]{amsproc}

\usepackage{graphicx}
\usepackage{amssymb}
\usepackage{amsmath}
\usepackage{amsthm}

\makeatletter
	
	\@addtoreset{equation}{section}
\makeatother

\newtheorem{thm}{Theorem}[section]
\newtheorem{prop}[thm]{Proposition}

\theoremstyle{definition}

\newtheorem{rem}[thm]{Remark}

\newcommand{\R}{\mathbb{R}}
\newcommand{\C}{\mathbb{C}}
\newcommand{\N}{\mathbb{N}}

\newcommand{\mean}[3]{\mathbb{E}^{#1}_{#2}\left[ #3 \right]}
\newcommand{\sle}{\mathrm{SLE}}
\newcommand{\skle}{\mathrm{SKLE}}
\newcommand{\bmd}{\mathrm{BMD}}
\newcommand{\lm}{\mathrm{LM}}
\newcommand{\disk}{\mathbb{D}}
\newcommand{\uhp}{\mathbb{H}}
\newcommand{\Slit}{\mathsf{Slit}}
\newcommand{\slit}{\mathbf{s}}
\DeclareMathOperator{\dist}{dist}
\DeclareMathOperator{\hcap}{hcap}

\title[LM in terms of SKLE in chordal case]%
{Reformulation of Laplacian-$b$~motion \\
in terms of \\
stochastic~Komatu--Loewner~evolution \\
in the chordal~case}

\author[T.\ Murayama]{Takuya Murayama}
\address{Department of Mathematics, Graduate School of Science, Kyoto University, Kyoto 606-8502, Japan.
(Research Fellow of Japan Society for the Promotion of Science)}
\email{murayama@math.kyoto-u.ac.jp}

\begin{document}

\maketitle

\begin{abstract}

We investigate the relationship between the Laplacian-$b$ motion and
stochastic Komatu--Loewner evolution (SKLE)
on multiply connected subdomains of the upper half-plane,
both of which are analogues to SLE.
In particular, we show that, if the driving function of an SKLE is given
by a certain stochastic differential equation,
then this SKLE is the same as a time-changed Laplacian-$b$ motion.
As application, we prove the finite time explosion of SKLE corresponding to
Laplacian-$0$ motion, or $\mathrm{SLE_6}$, in the sense that
the solution to the Komatu--Loewner equation for the slits blows up.
\end{abstract}

\section{Introduction}
\label{sec:intro}

Ever since Schramm~\cite{Sc00} introduced
the \emph{stochastic Loewner evolution with parameter $\kappa>0$}
(abbreviated as $\sle_{\kappa}$),
numerous studies have been conducted for identifying the scaling limits
of several two-dimensional lattice models in statistical physics,
such as loop-erased random walk, percolation exploration process and
critical Ising interface.
Most of these results are established on simply connected planar domains
such as the upper half-plane $\uhp=\{z \in \C; \Im z>0\}$,
since the definition of $\sle_{\kappa}$ depends on the theory of conformal maps
on simply connected domains, especially on the \emph{chordal Loewner equation}.
It is thus a non-trivial problem to extend $\sle_{\kappa}$
to multiply connected domains, and the way of extension is not unique.
For example, Lawler~\cite{La06} introduced the \emph{Laplacian-$b$ motion}
$\lm_b$ with $b=(6-\kappa)/(2\kappa)$
\footnote{Zhan~\cite{Zh04} considered a similar object independently,
where he called it the harmonic random Loewner chain and
stood in a viewpoint different from that of Lawler.}
as a candidate of the scaling limit of \emph{Laplacian-$b$ random walk}.
Its definition is based on the chordal Loewner equation and
the Maruyama--Girsanov transform
of the driving process associated with $\sle_{\kappa}$.
On the other hand, Chen and Fukushima~\cite{CF18} introduced
the \emph{stochastic Komatu--Loewner evolution} $\skle_{\alpha, \beta}$
as a process satisfying the \emph{domain Markov property} and
invariance under linear conformal maps,
both of which are typical properties of $\sle_{\kappa}$.
(Here, we use $\beta$ instead of $b$ in \cite[Eq.~(3.32)]{CF18}
to avoid a conflict with the exponent $b$ above.)
Its definition is motivated by Bauer and Friedrich~\cite{BF04, BF06, BF08} and
based on the \emph{chordal Komatu--Loewner equations}.

In this paper, we investigate the relationship between the two processes above
by means of Chen, Fukushima and Suzuki~\cite{CFS17} and the author~\cite{Mu18}.
After the definitions of $\skle_{\alpha, \beta}$ and $\lm_b$ are reviewed
in Sections~\ref{subsec:SKLE} and \ref{subsec:LM}, respectively,
we show in Section~\ref{subsec:reconst}
that $\lm_b$ can be regarded as $\skle_{\alpha, \beta}$
with appropriate $\alpha$ and $\beta$ modulo time-change.
In particular, the relation between $\skle_{\alpha, \beta}$ and $\lm_b$
becomes rather simple in the case $b=0$ ($\kappa=6$).
By using this identification, we prove in Section~\ref{subsec:expl}
that the finite time explosion of the solution
to the Komatu--Loewner equation for the slits~\cite[Eq.~(3.32) and (3.33)]{CF18}
corresponding to $\lm_0$.
This result provides a non-trivial example that satisfies the assumption
of \cite[Theorem~3.2]{Mu19+a}.
We note that, although each of the results in Section~\ref{sec:main} is simple,
they together illustrate a possible way to apply SKLE
to the theory of SLE on multiply connected domains.

\section{Preliminaries}
\label{sec:prel}

\subsection{Stochastic Komatu--Loewner evolution}
\label{subsec:SKLE}

Let us recall the notation in \cite{CF18, Mu18}. Let $N$ be a positive integer.
\begin{itemize}
\item $\Slit$ is the set of all elements
\begin{align*}
\slit&=(\slit_l)_{l=1}^{3N}
=(y_1, \ldots, y_N, x_1, \ldots, x_N, x^r_1, \ldots, x^r_N) \\
&\in (0,\infty)^N \times \R^{2N}
\end{align*}
with $x_j<x^r_j$ for each $j=1, \ldots, N$ such that
either $x^r_j<x_k$ or $x^r_k<x_j$ holds
if $y_j=y_k$ for two distinct numbers $j, k \in \{1, \ldots, N\}$.
\item $C_j(\slit)$ is the segment whose endpoints are $z_j=x_j+iy_j$
and $z^r_j=x^r_j+iy_j$ for $\slit \in \Slit$ and $1 \leq j \leq N$.
\item $D(\slit)$ is the \emph{standard slit domain}
$\uhp \setminus \bigcup_{j=1}^N C_j(\slit)$ for $\slit \in \Slit$.
\item $b_l \colon \R \times \Slit \to \R$, $1 \leq l \leq 3N$,
are the functions defined by
\[
b_l(\xi_0, \slit)=\begin{cases}
-2\pi \Im\Psi_{D(\slit)}(z_l, \xi_0) &(1 \leq l \leq N) \\
-2\pi \Re\Psi_{D(\slit)}(z_{l-N}, \xi_0) &(N+1 \leq l \leq 2N) \\
-2\pi \Re\Psi_{D(\slit)}(z^r_{l-2N}, \xi_0) &(2N+1 \leq l \leq 3N),
\end{cases}
\]
where $\Psi_D(\slit)$ is the \emph{complex Poisson kernel} of
\emph{Brownian motion with darning (BMD)} for $D(\slit)$~\cite[Lemma~4.1]{CFR16}.
\end{itemize}
We also recall that a set $F \subset \uhp$ is called a (compact $\uhp$-)\emph{hull}
if $F$ is bounded and relatively closed in $\uhp$
and if $\uhp \setminus F$ is still simply connected.
For any hull $F$ in a standard slit domain $D$,
there exists a unique conformal map $f_F$ from $D \setminus F$
onto another standard slit domain $\tilde{D}$,
which is called the \emph{canonical map},
such that the \emph{hydrodynamic normalization at infinity}
$f_F(z)=z+a/z+o(z^{-1})$ as $z \to \infty$ holds.
The positive constant $\hcap^D(F):=a$
is called the \emph{(BMD) half-plane capacity relative to $D$}.
See \cite[Proposition~2.3]{Mu18}.

Now, fix $t_0 \in (0, \infty]$.
Suppose that $a_t$ is a strictly increasing differentiable function
of $t \in [0, t_0)$ with $a_0=0$
and that $\xi(t)$ be an $\R$-valued continuous function on the same interval.
We consider the following ordinary differential equations:
\begin{align}
\frac{d}{dt}\slit_l(t)&=\frac{\dot{a}_t}{2}b_l(\xi(t), \slit(t)),& &1 \leq l \leq 3N,
\label{eq:KLs} \\
\frac{d}{dt}g_{t}(z)&=-\pi \dot{a}_t \Psi_{D(\slit(t))}(g_{t}(z),\xi(t)),&
&g_{0}(z)=z \in D(\slit(0)). \label{eq:KL}
\end{align}
Here, the dot on $a_t$ stands for the $t$-derivative.
We call \eqref{eq:KL} the chordal Komatu--Loewner equation~\cite{BF08, CFR16}
and \eqref{eq:KLs} the Komatu--Loewner equation for the slits~\cite{BF08, CF18}.
We use the symbol $\zeta$ to denote the explosion time of the solution $\slit(t)$
to \eqref{eq:KLs}.
Moreover, we define $F_t:=\{z \in D; t_z \leq t\}$ for $t \in [0, \zeta)$, where
$t_z:=\zeta \wedge \sup\{t>0; \lvert g_t(z)-\xi(t) \rvert>0\}$ for $z \in D=D(\slit(0))$.
Then $\{F_t\}_{t \in [0, \zeta)}$ is a family of \emph{continuously}
growing hulls~\cite[Definition~4.2]{Mu18} with $\hcap^D(F_t)=a_t$,
$g_t$ is the canonical map from $D \setminus F_t$ onto $D(\slit(t))$
for each $t \in [0, \zeta)$, and it follows that
\begin{equation} \label{eq:shrink}
\bigcap_{\delta>0}\overline{g_t(F_{t+\delta} \setminus F_t)}=\{\xi(t)\}.
\end{equation}
We call a function $\xi(t)$ satisfying \eqref{eq:shrink}
the \emph{driving function} of $\{F_t\}$.
See \cite[Section~5]{CF18} for the proof of these facts.
We remark that they are valid also in the case $N=0$,
except that \eqref{eq:KLs} does not appear.
Hence we may put $\zeta=\infty$.
In this case, the complex Poisson kernel is
\[
\Psi_{\uhp}(z, \xi_0)=-\frac{1}{\pi}\frac{1}{z-\xi_0},
\]
and \eqref{eq:KL} is called the chordal Loewner equation.

Let $\alpha$ be a non-negative function on $\Slit$ homogeneous with degree 0.
Here, a function $f(\slit)$ is said to be homogeneous with degree $\delta \in \R$
if $f(c\slit)=c^{\delta}f(\slit)$ holds for all $c>0$ and $\slit \in \Slit$.
Let $\beta$ be a function on $\Slit$ homogeneous with degree $-1$.
We further suppose that both $\alpha$ and $\beta$ are locally Lipschitz continuous.
$\skle_{\alpha, \beta}$ on a standard slit domain $D$~\cite[Section~5]{CF18}
is defined as the random continuously growing hulls $F_t$ in $D$
that are obtained via the procedure above with $a_t=2t$
and $\xi(t)$ given by the following stochastic differential equation (SDE)
\begin{equation} \label{eq:SKLE}
d\xi(t)=\alpha(\slit(t)-\widehat{\xi}(t))\,dB_t+\beta(\slit(t)-\widehat{\xi}(t))\,dt.
\end{equation}
Here, $\widehat{\xi}(t)$ stands for the $3N$-dimensional vector
whose first $N$ entries are zero and last $2N$ entries are $\xi(t)$.
In this case, we should regard \eqref{eq:KLs} and \eqref{eq:SKLE} together
as a system of SDEs, and $\zeta$ above is replaced by
the explosion time of the solution $(\xi(t), \slit(t))$ to this system.
$\sle_{\kappa}$ is a special case of this definition
where $N=0$, $\alpha=\sqrt{\kappa}$ and $\beta=0$.

Since it is sometimes convenient to regard $\alpha$ and $\beta$ in \eqref{eq:SKLE}
as functions on $\R \times \Slit$,
we introduce the notation $f(\xi_0, \slit):=f(\slit-\widehat{\xi_0})$
for a function $f$ on $\Slit$.
The function $f(\xi_0, \slit)$ so defined has
the \emph{invariance under horizontal translation}
$f(\xi_0, \slit)=f(0, \slit-\widehat{\xi_0})$.
Conversely, we define $\tilde{f}(\slit):=\tilde{f}(0, \slit)$
if a function $\tilde{f}$ on $\R \times \Slit$ has this invariance.

\subsection{Laplacian-$b$ motion}
\label{subsec:LM}

Let $Z^{\uhp}=(Z^{\uhp}_t, \mathbb{P}^{\uhp}_{z_0})$ be
an absorbing Brownian motion in $\uhp$ starting at $z_0 \in \uhp$
and $P^{\uhp}(t,z,dw)$ be the transition probability of $Z^{\uhp}$.
Doob's $h$-transform $\hat{Z}=(\hat{Z}_t, \hat{\mathbb{P}}_{z_0})$
with harmonic function $h(z)=\Im z$ is called
an \emph{$\uhp$-excursion}~(\cite[Section~5.3]{La05}, \cite[Section~3.5]{La06}).
In other words, $\hat{Z}$ is a strong Markov process with transition probability
\begin{equation} \label{eq:Hexc}
\hat{P}(t,z,dw):=\frac{\Im w}{\Im z}P^{\uhp}(t,z,dw), \quad t \geq 0,\ z \in \uhp.
\end{equation}
(See \cite{CW05} or \cite{Do84} for a general definition of $h$-transform.)
We can observe from \eqref{eq:Hexc}
that $\Re \hat{Z}$ and $\Im \hat{Z}$ are independent,
that $\Re \hat{Z}$ is just a standard Brownian motion
and that $\Im \hat{Z}$ is the $h$-transform
of a standard Brownian motion with $h(x)=x$.
By Chapter~VI, Section~3 of \cite{RY99}, it is a three-dimensional Bessel process.
Therefore, we can define an $\uhp$-excursion $\hat{Z}$
starting at a boundary point $z_0 \in \partial \uhp$ as well,
and the lifetime of $\hat{Z}$ is infinite $\hat{\mathbb{P}}_{z_0}$-almost surely.

Let $D':=\uhp \setminus \bigcup_{j=1}^N A_j$,
where $A_j \subset \uhp$, $1 \leq j \leq N$, are mutually disjoint,
compact continua with smooth boundaries.
Put
\[
Q(z, D'):=\hat{\mathbb{P}}_{z}
\left(\hat{Z}(0, \infty) \subset D'\right),\quad z \in D' \cup \partial \uhp.
\]
It then follows from \eqref{eq:Hexc} that
\begin{equation} \label{eq:Qformula}
Q(z, D')=1-\hat{\mathbb{P}}_{z}\left(\hat{\tau}_{D'}<\infty\right)
=1-\frac{\mean{\uhp}{z}{\Im Z^{\uhp}_{\tau_{D'}}; \tau_{D'}<\infty}}{\Im z},
\quad z \in D',
\end{equation}
where $\hat{\tau}_{D'}$ and $\tau_{D'}$ are
the exiting times from $D'$ of $\hat{Z}$ and $Z^{\uhp}$, respectively.
$Q(\xi_0, D')$ for $\xi_0 \in \partial \uhp$ is the limit of \eqref{eq:Qformula}
as $z \to \xi_0$.
We see later that $Q(\xi_0, D')$ is a differentiable function of $\xi_0 \in \R$.
In particular, we can consider $(\partial \log Q/\partial \xi_0)(\xi_0, D')$.

Based on the function $Q$,
the construction of $\lm_b$ for $b>-1/2$ is done as follows:
Let $\kappa:=6/(2b-1)$ and $(g^0_t, F^0_t)_{t \geq 0}$ be an $\sle_{\kappa}$
driven by $U(t)=\sqrt{\kappa}B_t$,
where $B=(B_t)_{t \geq 0}$ is a standard Brownian motion
on a filtered probability space
$(\Omega, \mathcal{F}_{\infty}, (\mathcal{F}_t)_{t \geq 0}, \mathbf{P})$
with the usual conditions.
We put $T_{D'}:=\inf\{t>0; F^0_t \nsubseteq D'\}$ and
$D'_t:=g^0_t(D' \setminus F^0_t)$,
and define the processes $X$ and $X^{(n)}$ for $n \in \N$ by
\[
X_t:=\frac{6-\kappa}{2\sqrt{\kappa}}
\frac{\partial \log Q}{\partial \xi_0}(U(t), D'_t)\mathbf{1}_{\{t<T_{D'}\}}
\quad \text{and}\quad X^{(n)}_t:=X_t \mathbf{1}_{\{t<T_n\}},
\]
where $T_n:=T_{D'} \wedge \inf\{t>0; \lvert X_t \rvert \geq n\}$.
We further define the exponential local martingale
\[
M^{(n)}_t:=\exp\left(\int_0^t X^{(n)}_s \,dB_s
-\frac{1}{2}\int_0^t \lvert X^{(n)}_s \rvert^2 \,ds\right)
\]
and the stopping times
$S_{n,m}:=T_n \wedge \inf\{t>0; \lvert M^{(n)}_t \rvert \geq m\}$, $m \in \N$.
The stopped local martingale
\[
M^{(n)}_{t \wedge S_{n,m}}=\exp\left(\int_0^{t \wedge S_{n,m}} X_s \,dB_s
-\frac{1}{2}\int_0^{t \wedge S_{n,m}} \lvert X_s \rvert^2 \,ds\right)
\]
is then bounded and thus a uniformly integrable martingale.
If we define the measure $\mathbf{Q}^{(n,m)}$ on $\mathcal{F}_{\infty}$ by
\[
\mathbf{Q}^{(n,m)}(A):=\mathbf{E}^{\mathbf{P}}
\left[\mathbf{1}_A \lim_{t \to \infty}M^{(n)}_{t \wedge S_{n,m}}\right],
\quad A \in \mathcal{F}_{\infty},
\]
then by the Maruyama--Girsanov theorem
\[
B^{(n,m)}_t:=B_t-\int_0^{t \wedge S_{n,m}} X_s \,ds,
\quad t \in [0, \infty),
\]
is a standard Brownian motion under the measure $\mathbf{Q}^{(n,m)}$.
In other words, the driving function $U(t)$ of $\{F^0_t\}$ satisfies
the following SDE under $\mathbf{Q}^{(n,m)}$:
\begin{equation} \label{eq:GirLM}
dU(t)=\frac{6-\kappa}{2}\frac{\partial \log Q}{\partial \xi_0}(U(t), D'_t)\,dt
+\sqrt{\kappa}\,dB^{(n,m)}_t,\quad 0 \leq t \leq S_{n,m}.
\end{equation}
(The law of) $\{F^0_t\}_{0 \leq t \leq S_{n,m}}$ under $\mathbf{Q}^{(n, m)}$
is regarded as $\lm_b$ stopped by $S_{n,m}$.
Motivated by \eqref{eq:GirLM}, Lawler~\cite{La06} defined $\lm_b$
as the random Loewner evolution driven by a solution to \eqref{eq:GirLM}
with $B^{(n,m)}$ replaced by a Brownian motion independent of $n$ and $m$.

\begin{rem} \label{rem:GirLM}
The exponent $b$ implicitly appears in \eqref{eq:GirLM}
in the sense that $(6-\kappa)/2=\kappa b$.
Since the chordal Loewner equation that we consider in this paper
is the linear time-change of \cite[Eq.\ (4.16)]{La06},
the SDE~\eqref{eq:GirLM} is the time-change of the original one
given in \cite[Section~4.4]{La06}.
The exponential martingale $M^{(n)}_t$ is originally obtained
in a different fashion as well.
\end{rem}

In this article, we only refer to Lawler~\cite[Section~4]{La06}
and do not mention the reason why the random evolution
defined as above is a candidate of the scaling limit of Laplacian-$b$ random walk.
However, let us comment that his argument is remarkable
in that we can extend $\sle_{\kappa}$ to multiply connected domains
by the Maruyama--Girsanov transformation.

\section{Relationship between LM and SKLE}
\label{sec:main}

\subsection{Reformulation of Laplacian-$b$ motion as SKLE}
\label{subsec:reconst}

In this section, we observe the relationship between LM and SKLE.
We take $\lm_b$ as our starting point.
Let $D'$ and $(g^0_t, F^0_t)$ be as in Section~\ref{subsec:LM}.
To move our domain from $D'$ to a standard slit domain
as in Section~\ref{subsec:SKLE},
we perform a preparatory transformation on $(g^0_t, F^0_t)$ as follows:
Let $h$ be a conformal map hydrodynamically normalized
from $D'$ to a standard slit domain $D$,
whose existence and uniqueness are ensured by \cite[Proposition~2.3]{Mu18}.
For each $t$, we denote the canonical map of the hull $F_t:=h(F^0_t)$
by $g_t \colon D \setminus F_t \to D_t$ and put $h_t:=g_t \circ h \circ (g^0_t)^{-1}$.
Then by \cite[Theorem~4.8]{Mu18},
$\{F_t\}_{t \leq S_{n,m}}$ is a family of continuously growing hulls in $D$,
the map $g_t$ satisfies \eqref{eq:KL}
with $\dot{a}_t=2h_t'(U(t))^2$ and $\xi(t)=h_t(U(t))$,
and the solution $\slit(t)$ to \eqref{eq:KLs} with these $a_t$ and $\xi(t)$
enjoys $D(\slit(t))=D_t$.
Moreover, it follows from \cite[Eq.~(4.7)]{Mu18} that
\begin{align}
d\xi(t) &= \biggl\{\frac{6-\kappa}{2}\left(h_t'(U(t))
	\frac{\partial \log Q}{\partial \xi_0}(U(t), D'_t)-h_t''(U(t))\right) \biggr.
	\label{eq:Ito_LM} \\
&\phantom{=}{}\biggl. -h_t'(U(t))^2 b_{\bmd}(\xi(t), \slit(t)) \biggr\} \,dt
	+\sqrt{\kappa}h_t'(U(t))\,dB^{(n,m)}_t
	\notag
\end{align}
for $t \leq S_{n,m}$ under the measure $\mathbf{Q}^{(n,m)}$.
Here, $b_{\bmd}$ is the \emph{BMD domain constant},
a locally Lipschitz function on $\R \times \Slit$
which is invariant under horizontal translation and homogeneous with degree~$-1$.
See Eq.~(6.1) and Lemma~6.1 of \cite{CF18}.
The aim of this subsection is to rewrite \eqref{eq:Ito_LM}
in the form of \eqref{eq:SKLE} and to check that its coefficients satisfy
the conditions that are required in Section~\ref{subsec:SKLE}
to define $\skle_{\alpha, \beta}$.

To deform the expression~\eqref{eq:Ito_LM} into a form independent of $D'_t$,
we utilize the conformal transformation rule~\cite[Eq.~(3.12)]{La06}
for $\partial \log Q/\partial \xi_0$:
\begin{equation} \label{eq:conftrans}
\frac{\partial \log Q}{\partial \xi_0}(U(t), D'_t)
=h_t'(U(t))\frac{\partial \log Q}{\partial \xi_0}(\xi(t), D_t)
	+\frac{h_t''(U(t))}{h_t'(U(t))}.
\end{equation}
Substituting \eqref{eq:conftrans} into \eqref{eq:Ito_LM} yields
\begin{align*}
d\xi(t) &= \left\{ \frac{6-\kappa}{2}
	\frac{\partial \log Q}{\partial \xi_0}(\xi(t), \slit(t))
	-b_{\bmd}(\xi(t), \slit(t)) \right\}h_t'(U(t))^2\,dt \\
	&\phantom{=}{}+\sqrt{\kappa}h_t'(U(t))\,dB^{(n,m)}_t, \quad t \leq S_{n,m},
\end{align*}
where $Q(\xi_0, \slit):=Q(\xi_0, D(\slit))$ for $\xi_0 \in \R$ and $\slit \in \Slit$.

We now reparametrize $\{F_t\}$ by the half-plane capacity relative to $D$,
that is, $\check{F}_t:=F_{a^{-1}(2(t \wedge \check{S}_{n,m}))}$
with $\check{S}_{n,m}:=a(S_{n,m})/2$.
By this time-change, we have $\hcap^D(\check{F}_t)=2t$
for all $t \in [0, \check{S}_{n,m})$.
All the other quantities that are reparametrized in the same manner
are indicated by adding check mark symbol as well.
Brownian motion $(\check{B}^{(n,m)}_t)_{t \geq 0}$ then exists
on some enlargement of the filtered probability space
$(\Omega, (\check{\mathcal{F}}_t)_{t \geq 0}, \mathbf{Q}^{(n,m)})$,
satisfying
\begin{align}
d\check{\xi}(t) &= \left\{\frac{6-\kappa}{2}
	\frac{\partial \log Q}{\partial \xi_0}(\check{\xi}(t), \check{\slit}(t))
	-b_{\bmd}(\check{\xi}(t), \check{\slit}(t)) \right\}dt \label{eq:norm_LM} \\
	&\phantom{=}{}+\sqrt{\kappa}\,d\check{B}^{(n,m)}_t,
	\quad t<\check{S}_{n,m}, \notag
\end{align}
by \cite[Theorem~V.1.7]{RY99}.

The expression~\eqref{eq:norm_LM} suggests the relation
between $\lm_b$ and $\skle_{\alpha, \beta}$.
Namely, if we can define $\skle_{\alpha, \beta}$ on $D$ with
\begin{equation} \label{eq:parameter}
\alpha(\slit):=\sqrt{\kappa}\quad \text{and} \quad \beta(\slit)
:=\frac{6-\kappa}{2}\frac{\partial \log Q}{\partial \xi_0}(0, \slit)-b_{\bmd}(0, \slit),
\end{equation}
and pull it back to $D'$ by the map $h \colon D' \to D$,
then the resulting random evolution is a time-changed $\lm_b$ on $D'$.
Indeed, we can do the reverse procedure, that is,
start at \eqref{eq:norm_LM} to get \eqref{eq:GirLM}
by a similar computation and enlargement of the underlying probability space
using the inverse map $h^{-1}$.
The remaining thing is thus to check $\alpha$ and $\beta$ of \eqref{eq:parameter}
satisfies the conditions in Section~\ref{subsec:SKLE}:

\begin{prop} \label{lem:LMLip}
$(\partial \log Q/\partial \xi_0)(\xi_0, \slit)$ is invariant
under horizontal translations, homogeneous with degree~$-1$
and locally Lipschitz continuous.
\end{prop}

\begin{proof}
The translation invariance and homogeneity with degree~$-1$
are obvious from the conformal transformation rule \eqref{eq:conftrans}.
We therefore prove only the local Lipschitz continuity.

We put $D:=D(\slit)$ and $C_j:=C_j(\slit)$ for $\slit \in \Slit$,
and denote by $K_D$ the Poisson kernel of absorbing Brownian motion in $D$.
We also denote by $\varphi_{\slit}^{(j)}(z)$ the harmonic measure of $C_j$.
In other words, $\varphi_{\slit}^{(j)}$ is a unique bounded harmonic function
on $D$ with boundary values $1$ on $C_j$
and $0$ on $\partial \uhp \cup \bigcup_{k \neq j} C_k$.
The quantity $Q(\xi_0, \slit)$ for $\xi_0 \in \partial \uhp$ is computed
by using \eqref{eq:Qformula} and $K_D$ as follows:
\begin{align}
Q(\xi_0, \slit)&=\lim_{z \to \xi_0}Q(z, D)
=1-\lim_{z \to \xi_0}
\frac{\mean{\uhp}{z}{\Im Z^{\uhp}_{\tau_D}; \tau_D<\infty}}{\Im z} \notag \\
&=1+\frac{\partial}{\partial \mathbf{n}_{\xi_0}}\sum_{j=1}^N\int_{\partial_p C_j}
	\Im w K_D(\xi_0, w) \,\lvert dw \rvert \notag \\
&=1+\sum_{j=1}^N y_j \frac{\partial}{\partial \mathbf{n}_{\xi_0}}
	\varphi_{\slit}^{(j)}(\xi_0). \label{eq:Q_ssd}
\end{align}
Here, $\mathbf{n}_{\xi_0}$ stands for the outward unit normal vector at $\xi_0$,
and $\partial_p C_j$ represents the boundary of $\uhp \setminus C_j$
in the path distance topology.
Namely, it consists of the left and right endpoints $z_j$ and $z^r_j$,
the upper side $C_j^+$ of the slit and the lower one $C_j^-$.
It follows from \eqref{eq:Q_ssd} that
\[
\frac{\partial \log Q}{\partial \xi_0}(\xi_0, \slit)
=\left(1+\sum_{j=1}^N y_j \frac{\partial}{\partial \mathbf{n}_{\xi_0}}
	\varphi_{\slit}^{(j)}(\xi_0)\right)^{-1}
	\sum_{j=1}^N y_j \frac{\partial}{\partial \xi_0}
	\frac{\partial}{\partial \mathbf{n}_{\xi_0}}\varphi_{\slit}^{(j)}(\xi_0).
\]
The function $\partial_{\mathbf{n}_{\xi_0}}\varphi_{\slit}^{(j)}(\xi_0)$
is locally Lipschitz by \cite[Eq.~(9.24)]{CFR16}.
We can mimic the argument of \cite[Section~9]{CFR16},
in which \cite[Eq.~(9.24)]{CFR16} was derived from \cite[Eq.~(9.12)]{CFR16},
to show that
$\partial_{\xi_0}\partial_{\mathbf{n}_{\xi_0}}\varphi_{\slit}^{(j)}(\xi_0)$
is also locally Lipschitz.
We thus reach the desired conclusion.
\end{proof}

By Proposition~\ref{lem:LMLip}, we can define $\skle_{\alpha, \beta}$
with $\alpha$ and $\beta$ given by \eqref{eq:parameter} and
thus obtain a time-changed $\lm_b$ by pulling it back to $D'$.
If $D'$ itself is a standard slit domain,
then $\skle_{\alpha, \beta}$ is exactly the same as
the time-changed $\lm_b$ on $D'$.

\subsection{Explosion time of SKLE corresponding to $\lm_b$}
\label{subsec:expl}

In this subsection, we discuss the exit of $\lm_b$ from the domain $D'$.
Our problem is whether $\lm_b$ exit from $D'$, i.e.,
the exit time $T_{D'}$ in Section~\ref{subsec:LM} is finite or not.
We are interested in the behavior of $\lm_b$ around $T_{D'}$ as well.

It is commented in \cite[Section~4.6]{La06} that
$\lm_b$ with $b=(6-\kappa)/(2\kappa)$ for $0 < \kappa \leq 4$
is not likely to exit the domain $D'$, that is,
$T_{D'}$ in Section~\ref{subsec:LM} should be infinite.
This observation is partly based on the fact that
$\sle_{\kappa}$ with $0 < \kappa \leq 4$ is a simple curve
with probability one \cite[Proposition~6.9]{La05}.
However, the proof of the property $T_{D'}=\infty$ has not been known so far.
One of the difficulties is that
there may not exist a single probability measure $\mathbf{Q}$
on $(\Omega, \mathcal{F}_{\infty})$
under which \eqref{eq:GirLM} holds for all $t \in [0, T_{D'})$
with $B^{(n,m)}$ replaced by a Brownian motion independent of $n$ and $m$.
The problem itself may be ``ill-posed'' unless such a measure $\mathbf{Q}$ exists.
We note that, since the SDE~\eqref{eq:GirLM} is not closed
with respect to the unknown variable $U(t)$
but contains another unknown variable $D'_t$,
the construction of $\mathbf{Q}$ via this SDE is not straightforward.

In what follows, the explosion problem of $\skle_{\alpha, \beta}$
with $\alpha$ and $\beta$ given by \eqref{eq:parameter},
which is always ``well-posed'', is addressed
instead of the original exit problem of $\lm_b$.
In Section~\ref{subsec:SKLE}, $\skle_{\alpha, \beta}$ is defined only up to
the explosion time $\zeta$ of the solution $W_t=(\xi(t), \slit(t))$
to the system of SDEs \eqref{eq:KLs} and \eqref{eq:SKLE}.
As described in \cite[Section~1]{Mu19+a}, $\zeta$ should correspond to
the ``exit time'' of $\skle_{\alpha, \beta}$ from the standard slit domain $D$.
On the basis of this observation, the author obtained
the following asymptotic behavior of the solution to \eqref{eq:KLs} around $\zeta$
in \cite[Section~3]{Mu19+a}:
Recall that
\[
R(\xi_0, \slit):=\min_{1 \leq j \leq N}\dist(C_j(\slit), \xi_0)
\]
is a function on $\R \times \Slit$
having the invariance under horizontal translation.
We say that a function \emph{$f$ on $\Slit$ enjoys Condition~{\rm (B')}}
if $f(\slit)$ is bounded on the set $\{\slit \in \Slit; R(\slit)=R(0, \slit)>r\}$
for every $r>0$.
The previous result~\cite[Theorem~3.2]{Mu19+a} asserts that,
under Condition~(B') on the coefficients in \eqref{eq:SKLE},
$\lim_{t \to \zeta} R(W_t) = 0$ holds a.s.\ on the event $\{ \zeta < \infty \}$.
A intuitive picture of this result is that,
if $\lim_{t \to \zeta} \dist(C_j(\slit(t)), \xi(t))=0$,
then the hull $F_t$ should ``exit'' $D$ from the corresponding slit $C_j$.

Now, we check that Condition~(B') holds in the case of $\lm_b$.
Since $\alpha(\slit)=\sqrt{\kappa}$ clearly satisfies this condition,
we only have to prove it for $\beta$ of \eqref{eq:parameter}:

\begin{prop} \label{lem:LMbdd}
$(\partial \log Q/\partial \xi_0)(0, \slit)$ satisfies Condition {\rm (B')}.
\end{prop}

\begin{proof}
By the scaling property that is derived from \eqref{eq:conftrans},
it suffices to prove that $(\partial \log Q/\partial \xi_0)(0, \slit)$
is bounded when $R(\slit)>1$.
By the strong Markov property of the $\uhp$-excursion $\hat{Z}_t$,
we have
\[
Q(z, D)=\hat{\mathbb{E}}_z\left[Q(\hat{Z}_{\hat{\sigma}_{\disk_+}}, D)\right]
=\int_0^{\pi}Q(e^{i\theta}, D)K_{\disk_+}(z, e^{i\theta})
	\frac{\sin\theta}{\Im z}\,d\theta
\]
for $\lvert z \rvert < 1$.
Here, $\hat{\sigma}_{\disk_+}$ denotes the hitting time of $\hat{Z}$
to $\disk_+ =\disk \cap \uhp$.
Letting $z \to \xi_0 \in \partial \uhp$ yields that
\[
Q(\xi_0, \slit)=-\int_0^{\pi}Q(e^{i\theta}, D)\frac{\partial}{\partial \mathbf{n}_{\xi_0}}
	K_{\disk_+}(\xi_0, e^{i\theta})\sin\theta \,d\theta.
\]
Hence we have
\[
\frac{\partial \log Q}{\partial \xi_0}(\xi_0, \slit)
=\frac{\frac{\partial}{\partial \xi_0}Q(\xi_0, \slit)}{Q(\xi_0, \slit)}
=\frac{\int_0^{\pi}Q(e^{i\theta}, D)\frac{\partial}{\partial \xi_0}
\frac{\partial}{\partial \mathbf{n}_{\xi_0}}K_{\disk_+}(\xi_0, e^{i\theta})\sin\theta \,d\theta}
{\int_0^{\pi}Q(e^{i\theta}, D)\frac{\partial}{\partial \xi_0}
K_{\disk_+}(\xi_0, e^{i\theta}) \sin\theta \,d\theta}.
\]
Computing the Poisson kernel $K_{\disk_+}$ gives the following expression:
\[
\frac{\partial \log Q}{\partial \xi_0}(0, \slit)
=\frac{2\int_0^{\pi}Q(e^{i\theta},D)\sin^2\theta\cos\theta \,d\theta}
	{\int_0^{\pi}Q(e^{i\theta},D)\sin^2\theta \,d\theta}.
\]
The right-hand side is bounded by two.
\end{proof}

Since we have established Condition~(B'),
it is natural in view of the intuitive picture above that
we interpret the exit problem of $\lm_b$ as follows:
$\zeta=\infty$ holds a.s.\ if and only if $0<\kappa\leq 4$.
However, this conjecture is still difficult to prove for general values of $\kappa$.
In what follows, we look only at the special case $\kappa=6$ (i.e., $b=0$).
The situation becomes much simpler in this case
since the drift term of \eqref{eq:GirLM} vanishes.
Hence $\lm_0$ is just $\sle_6$, and
especially no change of measures is necessary to obtain $\lm_0$.
The reconstruction procedure in Section~\ref{subsec:reconst}
thus provides another proof of \cite[Theorem~4.2]{CFS17},
which shows that the law of $\sle_6$ coincides with $\skle_{\sqrt{6}, -b_{\bmd}}$
modulo time-change until it exits $D$.
The following result follows from the fact that the $\sle_6$ hull
is space-filling with probability one~\cite[Proposition~6.10]{La05}:

\begin{thm} \label{prop:LM_0}
The explosion time $\zeta$ of the solution $W_t=(\xi(t), \slit(t))$
to \eqref{eq:KLs} and \eqref{eq:SKLE} whose coefficients are
$\alpha(\slit)=\sqrt{6}$ and $\beta(\slit)=-b_{\bmd}(\slit)$
is finite with probability one.
\end{thm}

\begin{proof}
Let $\{F_t\}_{t \in [0, \infty)}$ be an $\sle_6$ hull
defined on a filtered probability space
$(\Omega, (\mathcal{F}_t)_{t \geq 0}, \mathbf{Q})$ with the usual conditions.
\cite[Proposition~6.10]{La05} asserts that
$\mathbf{Q}\left(\bigcup_{t \in [0, \infty)}F_t=\uhp\right)=1$,
and hence $T_D:=\inf\{t>0; F_t \nsubseteq D\}<\infty$ $\mathbf{Q}$-a.s.
We define
\[
a_t:=\hcap^D(F_t),\quad \check{T}_D:=\frac{a(T_D-)}{2}\quad \text{and}\quad
\check{F}_t:=F_{a^{-1}(2(t \wedge \check{T}_D)}
\]
$\{\check{F}_t\}_{t<\check{T}_D}$ is then
the $\skle_{\sqrt{6}, -b_{\bmd}}$ hull in $D$
defined on an enlargement of $(\Omega, (\mathcal{F}_t)_{t \geq 0}, \mathbf{Q})$.
The pair of associated driving function $\check{\xi}(t)$
and slit vector $\check{\slit}(t)$
is thus a weak solution of the SDEs \eqref{eq:KLs} and \eqref{eq:SKLE}
whose coefficients are the above $\alpha$ and $\beta$ up to $\check{T}_D$.
The stopping time $\check{T}_D$ is its explosion time,
because if the solution could be continued, then
$F_{T_D}$ were included by $D$, which contradicts to the definition of $T_D$.

Because of the uniqueness in law of solutions
to the SDEs \eqref{eq:KLs} and \eqref{eq:SKLE},
it suffices to show that
$\check{T}_D$ is finite $\mathbf{Q}$-a.s.\ for completing the proof.
We now take a sample $\{F_t\}$ such that $T_D<\infty$.
Since $F_{T_D}$ is then bounded,
there exists $R>0$ such that $F_{T_D} \subset B(0,R) \cap D$.
By \cite[Eq.~(A.20)]{CF18}, we have
\begin{align*}
2\check{T}_D&=\hcap^D(F_{T_D}) \\
	&\leq \frac{2R}{\pi}\int_0^{\pi}\mean{*}{Re^{i\theta}}
	{\Im Z^*_{\sigma^*_{B(0,R) \cap D}}; \sigma^*_{B(0,R) \cap D}<\infty}
	\sin\theta \,d\theta \\
	&\leq \frac{4R^2}{\pi}<\infty.
\end{align*}
As $T_D<\infty$ holds $\mathbf{Q}$-a.s., we reach the desired conclusion.
\end{proof}

We have two remarks on Theorem~\ref{prop:LM_0}.
First, this theorem is highly non-trivial from the form of SDE~\eqref{eq:SKLE}
while the proof is straightforward in terms of $\sle_6$.
Second, it is clear that the proof of Theorem~\ref{prop:LM_0} does not work
for $\kappa \neq 6$.
We have to deal with the quantity $\partial \log Q/\partial \xi_0$
and reveal how much repulsive force is exerted between
the slit vector $\slit(t)$ and driving function $\xi(t)$.
This is yet to be investigated.

\section*{Acknowledgements}
This work was supported by JSPS KAKENHI Grant Number \linebreak JP19J13031.


\begin{thebibliography}{00}

\bibitem{BF04} R.\ O.\ Bauer and R.\ M.\ Friedrich, Stochastic Loewner evolution in multiply connected domains, C.\ R.\ Acad.\ Sci.\ Paris, Ser.\ I {\bf 339} (2004), 579--584.
\bibitem{BF06} R.\ O.\ Bauer and R.\ M.\ Friedrich, On radial stochastic Loewner evolution in multiply connected domains, J.\ Funct.\ Anal.\ {\bf 237} (2006), 565--588.
\bibitem{BF08} R.\ O.\ Bauer and R.\ M.\ Friedrich, On chordal and bilateral SLE in multiply connected domains, Math.\ Z.\ {\bf 258} (2008), 241--265.
\bibitem{CF18} Z.-Q.\ Chen and M.\ Fukushima, Stochastic Komatu--Loewner evolutions and BMD domain constant, Stochastic Process.\ Appl.\ {\bf 128} (2018), 545--594.
\bibitem{CFR16} Z.-Q.\ Chen, M.\ Fukushima and S.\ Rohde, Chordal Komatu--Loewner equation and Brownian motion with darning in multiply connected domains, Trans.\ Amer.\ Math.\ Soc.\ {\bf 368} (2016), 4065--4114.
\bibitem{CFS17} Z.-Q.\ Chen, M.\ Fukushima and H.\ Suzuki, Stochastic Komatu--Loewner evolutions and SLEs, Stochastic Process.\ Appl.\ {\bf 127} (2017), 2068--2087.
\bibitem{CW05} K.\ L.\ Chung and J.\ B.\ Walsh, {\it Markov processes, Brownian motion, and time symmetry}, 2nd ed., Grundlehren der mathematischen Wissenschaften, vol.\ 249, Springer, 2005.
\bibitem{Do84} J.\ L.\ Doob, {\it Classical Potential Theory and Its Probabilistic Counterpart}, Grundlehren der mathematischen Wissenschaften, vol.\ 262, Springer-Verlag, London, 1984.
\bibitem{La05} G.\ F.\ Lawler, {\it Conformally Invariant Processes in the Plane}, Mathematical Surveys and Monographs, vol.\ 114, American Mathematical Society, Providence, RI, 2005.
\bibitem{La06} G.\ F.\ Lawler, The Laplacian-$b$ random walk and the Schramm--Loewner evolution, Illinois J.\ Math.\ {\bf 50} (2006), 701--746.
\bibitem{Mu18} T.\ Murayama, Chordal Komatu--Loewner equation for a family of continuously growing hulls, Stochastic Process.\ Appl.\ {\bf 129} (2019), 2968--2990.
\bibitem{Mu19+a} T.\ Muramaya, On the slit motion obeying chordal Komatu--Loewner equation with finite explosion time, J.\ Evol.\ Equ.\ (2019), https://doi.org/10.1007/s00028-019-00519-3
\bibitem{RY99} D.\ Revuz and M.\ Yor, {\it Continuous Martingales and Brownian motion}, 3rd ed., Grundlehren der mathematischen Wissenschaften, vol.\ 293, Springer-Verlag, Berlin Heidelberg, 1999.
\bibitem{Sc00} O.\ Schramm, Scaling limits of loop-erased random walks and uniform spanning trees, Israel J.\ Math.\ {\bf 118} (2000), 221--288.
\bibitem{Zh04} D.\ Zhan, Random Loewner chains in Riemann surfaces, Ph.\ D.\ thesis, California Institute of Technology, 2004.

\end{thebibliography}
\end{document}